\documentclass[preprint,11pt,fleqn]{elsarticle}
\usepackage{ulem}
\usepackage{amsfonts}
\usepackage{graphicx}
\usepackage{amsthm}
\usepackage{mathrsfs}
\usepackage{multirow}
\usepackage{amsmath, mathtools, enumerate}
\usepackage{color}
\usepackage{booktabs}
\usepackage{epstopdf}
\usepackage{dsfont}

\usepackage{subfigure}
\allowdisplaybreaks
\numberwithin{equation}{section}
\numberwithin{table}{section}
\numberwithin{figure}{section}

\newtheorem{corollary}{Corollary}[section]
\newtheorem{lemma}{Lemma}[section]

\newtheorem{remark}{Remark}[section]
\newtheorem{example}{Example}[section]

\newcommand{\IR}{{\mathbb{R}}}

\newcommand{\mO}{{\mathcal{O}}}

\newcommand{\tT}{\intercal}

\usepackage{caption}
\usepackage{amssymb}

\journal{Elervier}
\begin{document}

\begin{frontmatter}



\title{A preconditioner based on sine transform for two-dimensional Riesz space fractional diffusion equations in convex domains}
\author{Xin Huang \fnref{inad2}}
\ead{hxin.ning@qq.com}
 \author{Hai-Wei Sun \corref{cor2}\fnref{inad2}}
 \ead{HSun@um.edu.mo}
 \cortext[cor2]{Corresponding author.}
 \address[inad2]{Department of Mathematics, University of Macau,
Macao.}



\begin{abstract}
In this paper, we develop a fast numerical method for solving the time-dependent Riesz space fractional diffusion equations with a nonlinear source term  in the convex domain. An implicit finite difference method is employed to discretize the Riesz space fractional diffusion equations with a penalty term in a rectangular region by the volume-penalization approach. The stability and the convergence of the proposed method are studied. As the coefficient matrix is with the Toeplitz-like structure, the generalized minimum residual method with a preconditioner based on the sine transform is exploited to solve the discretized linear system, where the preconditioner is constructed in view of the combination of two approximate inverse ${\tau}$ matrices, which can be diagonalized by the sine transform. The spectrum of the preconditioned matrix is also investigated.
 Numerical experiments are carried out to demonstrate the efficiency of the proposed method.
\end{abstract}


\end{frontmatter}

\section{Introduction}\label{sec:int}
Consider the following two-dimensional Riesz space fractional diffusion equations (RSFDEs) defined in convex domains with the homogenous Dirichlet boundary condition \cite{CLTA-ANM-2018}
\begin{align}\label{model_equation}
& \frac{\partial u}{\partial t}=k_x\frac{\partial ^{\alpha_1}u}{\partial|x|^{\alpha_1}}+
k_y\frac{\partial ^{\alpha_2}u}{\partial|y|^{\alpha_2}}+f(u,x,y,t), \ && (x,y,t)\in\Omega\times (0,T], \\ \label{initial}
& u(x,y,0)=u_0(x,y), \ && (x,y)\in \Omega, \\ \label{boundary}
&u(x,y,t)=0, \ && (x,y,t)\in \partial\Omega\times (0,T],
\end{align}
where $\Omega\in\IR^2$ is a convex region whose left and right boundaries are $a_1(y)$ and $b_1(y)$, while the lower and upper boundaries are $a_2(x)$ and $b_2(x)$, respectively, $k_x$ and $k_y$ are positive constants representing the diffusivity coefficients, the nonlinear source term $f(u,x,y,t)$ is supposed to hold the Lipschitz condition with respect to $u$ and $t$, the Riesz fractional derivatives $\frac{\partial ^{\alpha_1}u}{\partial|x|^{\alpha_1}}$ and $\frac{\partial ^{\alpha_2}u}{\partial|y|^{\alpha_2}}$ concerning to ${\alpha}_i\ (1<\alpha_i<2,\, i=1,2)$ are described by
$$
\frac{\partial ^{\alpha_1}u}{\partial|x|^{\alpha_1}}=c_{\alpha_1}
\left(_{a_1(y)}{\rm D}_{x}^{\alpha_1}u+_x{\rm D}_{b_1(y)}^{\alpha_1}u\right),\quad
\frac{\partial ^{\alpha_2}u}{\partial|y|^{\alpha_2}}=c_{\alpha_2}
\left(_{a_2(x)}{\rm D}_{y}^{\alpha_2}u+_y{\rm D}_{b_2(x)}^{\alpha_2}u\right),
$$
where $c_{\alpha_i}=-\frac{1}{2\cos(\alpha_i\pi/2)}>0$,  and the above left and right Riemann-Liouville fractional derivatives are depicted as
$$
_{a_1(y)}{\rm D}_{x}^{\alpha_1}u(x,y,t)=\frac{1}{\Gamma(2-\alpha_1)}\frac{\partial^2}{\partial x^2}\int_{a_1(y)}^{x}\frac{u(s,y,t)}{(x-s)^{\alpha_1-1}}{\rm d}s,
$$
$$
_x{\rm D}_{b_1(y)}^{\alpha_1}u(x,y,t)=\frac{1}{\Gamma(2-\alpha_1)}\frac{\partial^2}{\partial x^2}\int_{x}^{b_1(y)}\frac{u(s,y,t)}{(s-x)^{\alpha_1-1}}{\rm d}s.
$$
Analogously, $_{a_2(x)}{\rm D}_{y}^{\alpha_2}u$ and $ _y{\rm D}_{b_2(x)}^{\alpha_2}u$ can be defined in the same way.

Fractional differential equations provide plenty of efficient and powerful models for describing complex phenomena arising in various fields such as science, engineering, and physics; see \cite{BWM-WRS-2000,CLZ-PP-2001,GDM-JMAA-2008}. The Riesz fractional diffusion equation as one of the most popular fractional differential equations has received tremendous attention on practical applications, including the nonlocal dynamics \cite{Tv-HEP-2010}, the lattice model \cite{DZ-CMA-2012}, the FizHugh-Nagumo model \cite{FR-BJ-1961} and so on. The potential significance of studying Riesz fractional diffusion equations has attracted lots of researchers to study.

 Since the analytical solutions of fractional differential equations are usually unavailable, thus, it becomes a major research direction to seek the numerical solutions that has been intensively developed; see \cite{LAT-JCAM-2004,MST-JCP-2006,MT-ANM-2006}.
 Because of the Riesz fractional operator being with the nonlocal characteristics, the resulting coefficient matrix stemming from the numerical discretization is usually dense. Therefore, it needs $\mO(N^3)$ computational cost and $\mO(N^2)$ storage requirement to invert by the direct solvers, where $N$ denotes the number of unknowns. To overcome those demerits, numerous strategies are proposed for fast solving fractional diffusion equations with cheaper memory; see \cite{LS-JCP-2013,PS-JCP-2012,WWS-JCP-2010}.

However, if the fractional diffusion equations are defined in  convex domains, the resulting coefficient matrices are no longer with tensor forms. Hence the methods  that are available for the rectangular domains may not be extended  to solve the   problems in the convex domains. In 2018, Chen et al. in \cite{CLTA-ANM-2018} employed the alternating direction implicit (ADI) method to solve the RSFDEs \eqref{model_equation}--\eqref{boundary}  in the convex domains. In their work, the resulting coefficient matrix is consist of Toeplitz blocks  with different sizes. The preconditioner conjugate gradient (PCG) method with a circulant preconditioner is exploited to solve the   Toeplitz block system. Likewise, Jia and Wang \cite{JW-CMA-2018} adopted the same way to solve the distributed-order space-fractional diffusion equation defined on convex domains; see \cite{FL-AML-2018,LNS-JCP-2017}. Generally speaking, despite the total amount of calculation arising from the problem in convex domain is consistent in the order of magnitude comparing with the one of rectangular domains, its computational complexity   is more than that of the cases in  the rectangular domains.

To avoid the inconvenience of storing and inverting the coefficient matrices with different Toeplitz blocks, more effective and robust approaches are urgently desired. Jia and Wang \cite{JW-JCP-2016} firstly applied the volume-penalization method  \cite{ABF-NM-1999} to solve the fractional differential equations defined in convex domains. More precisely, they enlarge the   convex domain to   a rectangular domain, and thereby the resulting coefficient matrix is equivalent to the sum of an accurate tensor form of a Toeplitz-like matrix and a diagonal matrix which derives from the penalty term. As a result, the methods designed for the problems in rectangular domains are available for that in the convex domains. As a matter of fact, the essence of the volume-penalization is transforming the problem in the convex domain to   a more general problem. Numerically speaking,  for many problems, the solutions of the penalized fractional differential equations converge to the solutions of the original fractional differential equations, and such  facts have been confirmed in \cite{ABF-NM-1999,CF-ADE-2013,KNS-ANM-2015}. Consequently, the preconditioning techniques have been developed for those penalized systems; see  \cite{DSW-CAM-2019,CDL-JSC-2019}.

In this paper, we study the RSFDEs defined in convex domains combining with the volume-penalization strategy. We discretize the model equations with the penalty term on a rectangular region via an implicit finite difference method, which has been proved to be unconditionally stable and with first order in temporal and spatial direction. Note that the resulting coefficient matrix is the sum of a diagonal matrix,  whose entries equal to $1$ or $0$, and an exact tensor form of Toeplitz-like matrix. Therefore, the coefficient matrix can be stored  in $\mO(N)$ memory and the complexity of matrix-vector multiplication is of $\mO(N\log N)$ by the fast Fourier transformation (FFT). Thus, when the Krylov subspace method is exploited to solve the discretized linear system, the computational cost per iteration can keep $\mO(N\log N)$ operations. In order to speed up the convergent rate of the iterative method, constructing an efficient and feasible preconditioner is necessary and significant. To this end, we establish two step approximations to construct the preconditioner. Firstly, the sine transform based preconditioner, which is also called $\tau$ matrix that can be diagonalized by the discrete sine transform \cite{BB-Conference-1990,LFS-2020}, is applied to approximate the Toeplitz matrix. Afterward, we obtain a preliminary preconditioner   that the spectrum of the preconditioned matrix is proved to be uniformly bounded in the open interval $(1/2,3/2)$.
 Nevertheless, such preliminary preconditioner cannot be diagonalized by the sine transform matrix as its diagonal is not constant. In other words,   the inverse of the preconditioner cannot be calculated in $\mO(N\log N)$ operations. Therefore, we resort the strategy that has been adopted in \cite{CDL-JSC-2019,DSW-CAM-2019} to overcome this weakness. The second step to construct an efficient preconditioner is   the combination of two inverse $\tau$ matrices to approximate the previous preconditioner. The spectra of the preconditioned matrix are also analysed. Numerical results fully exhibit the efficiency of the proposed method.

The rest of the paper is organized as follows. In Section \ref{FDM}, an implicit finite difference method is employed to discretize the RSFDEs in the generalized rectangular domain. In Section \ref{S_AND_C}, the convergence and the stability of the difference scheme are studied. A preconditioner based on the sine transform is constructed in Section \ref{implement} and the spectra of the preconditioned matrix are discussed as well. In Section \ref{numerical results}, numerical results are reported to demonstrate the effectiveness of the proposed method. Concluding remarks are given in Section \ref{conclusion}.

\section{Finite difference discretization}\label{FDM}
To seek the numerical solution of the problem (\ref{model_equation})--\eqref{boundary}, we  extend the convex domain to a rectangular one that follows the idea  as shown in \cite{DSW-CAM-2019}.

Suppose that the domain ${\Omega}$ is contained in a rectangular domain $\bar{\Omega}=(a,b)\times(c,d)\supset\Omega$. Then, we reformulate the problem \eqref{model_equation}--\eqref{boundary} to be the following equations
\begin{align}\label{M1}
& \frac{\partial u_{\eta}}{\partial t}-k_x\frac{\partial ^{\alpha_1}u_{\eta}}{\partial|x|^{\alpha_1}}-
k_y\frac{\partial ^{\alpha_2}u_{\eta}}{\partial|y|^{\alpha_2}}+\frac{1-1_{\Omega}(x,y)}{\eta}u_{\eta}=\hat{f}(u_{\eta},x,y,t), \ && (x,y,t)\in\bar\Omega\times (0,T], \\ \label{I1}
& u_{\eta}(x,y,0)=\hat{u}_0(x,y), \ && (x,y)\in \bar\Omega, \\ \label{B1}
&u_{\eta}(x,y,t)=0, \ && (x,y,t)\in \partial\bar\Omega\times (0,T],
\end{align}
where $1_{\Omega}(x,y)$ is an indicator function satisfying $1_{\Omega}(x,y)=1$ if $(x,y)\in{\Omega}$, or $0$ elsewhere, $\hat f(u_\eta,x,y,t)$ and $\hat u_0(x,y)$ are the extensions of the source term $f(u,x,y,t)$ and the initial value $u_0(x,y)$ from $\Omega$ to $\bar\Omega$, respectively. For convenience, zero extensions for $\hat f(u_\eta,x,y,t)$ are used in the theoretical analysis and numerical experiments. Hence, the RSFDEs in \eqref{model_equation} will reduce to \eqref{M1} if $(x,y)\in\Omega$.
Moreover, the solution $u_{\eta}$ is supposed to satisfy the homogeneous Dirichlet boundary condition on the  extended region $\bar\Omega\setminus{\Omega}$; i.e.,
$$
{\lim_{\eta\to0^{+}}}u_{\eta}(x,y,t)=0.
$$
Indeed, it is have been proved that the solution of the penalized equation will converge to the solution of the original equation for many problems; see \cite{ABF-NM-1999,KNS-ANM-2015}. Therefore, in the following, we focus on the discretization of the penalized equations.

Let $n_1$, $n_2$, $m$ be positive integers. Denote $h_x=\frac{b-a}{n_1+1}$ and $h_y=\frac{d-c}{n_2+1}$ be the mesh sizes in $x$ direction and $y$ direction. We further define uniform spatial partitions as
$x_i=a+ih_x\ {\rm for}\ i=0,\dots,n_1+1$ and $ y_j=c+jh_y\ {\rm for}\ j=0,\dots,n_2+1$, respectively. Let $\Delta t=\frac{T}{m}$ and $t_k=k{\Delta}t$ for $k=0,\dots,m$.
In order to discretize the equation \eqref{M1}, we assume that the problem \eqref{M1}--\eqref{B1} is uniquely solvable and the solution $u_\eta(x,y,t)$ is sufficiently smooth on $\bar\Omega$.
Then, applying the shifted Gr$\ddot{\rm u}$nwald Letnikov difference scheme to approximate the left and right Riemann-Liouville fractional derivatives with respect to $x$ at grid point $(x_i,y_j,t_k)$, we obtain
\begin{align}
&_{a}{\rm D}_x^{\alpha_1}u_\eta(x_i,y_j,t_k)=\frac{1}{h_x^{\alpha_1}}\sum_{l=0}^{i+1}g_{l}^{(\alpha_1)}
u_\eta(x_{i-l+1},y_j,t_k)+\mO(h_x),\label{left-RL} \\
&_{x}{\rm D}_b^{\alpha_1}u_\eta(x_i,y_j,t_k)=\frac{1}{h_x^{\alpha_1}}\sum_{l=0}^{n_1-i+2}g_{l}^{(\alpha_1)}
u_\eta(x_{i+l-1},y_j,t_k)+\mO(h_x), \label{right-RL}
\end{align}
where the coefficients $g_{l}^{(\alpha_1)}$ are defined by
\begin{equation}\label{g_form}
g_0^{(\alpha_1)}=1, \quad g_{l}^{(\alpha_1)}=\left(1-\frac{\alpha_1+1}{l}\right)g_{l-1}^{(\alpha_1)}, \ {\rm for}\ l\ge1.
\end{equation}
Similarly, the above approximation results hold for the Riesz fractional derivative in $y$ direction.

Next, we consider the discretization of the time derivative term. The backward Euler difference scheme is exploited to approximate the time derivative as follows
\begin{equation}\label{time-discretization}
\frac{\partial u_\eta(x_i,y_j,t_k)}{\partial t}=\frac{u_\eta(x_i,y_j,t_k)-u_\eta(x_i,y_j,t_{k-1})}{\Delta t}+\mO(\Delta t).
\end{equation}
According to the above assumption and \eqref{time-discretization}, we derive that there exists a constant $c_1$ such that $u_\eta(x,y,t)$ satisfying
$$
|u_\eta(x_i,y_j,t_k)-u_\eta(x_i,y_j,t_{k-1})|\leq c_1\Delta t.
$$
For dealing with the nonlinear term, noting that $\hat f(u_{\eta},x,y,t)$ satisfies the Lipschitz condition for $u_{\eta}$ and $t$ in $\Omega$ and making using of the zeros extensions on the extended region, we derive that, for arbitrary $u_1$ and $u_2$, there exists a constant $c_2>0$ subjecting to
$$
|\hat f(u_1,x,y,t)-\hat f(u_2,x,y,t)|<c_2|u_1-u_2|.
$$
Likewise, denote $c_3>0$ be the Lipschitz constant concerning $t$. Then, it holds that
$$
|\hat f(u_\eta,x,y,t_1)-\hat f(u_\eta,x,y,t_2)|\leq c_3|t_1-t_2|.
$$
Thus, for $1\le i\le n_1$, $1\le j\le n_2$ and $1\le k\le m$, we have
\begin{align*}
&|\hat f(u_\eta(x_i,y_j,t_k),x_i,y_j,t_k)-\hat f(u_\eta(x_i,y_j,t_{k-1}),x_i,y_j,t_{k-1})| \\
\leq &|\hat f(u_\eta(x_i,y_j,t_k),x_i,y_j,t_k)-\hat f(u_\eta(x_i,y_j,t_k),x_i,y_j,t_{k-1})|\\
&+|\hat f(u_\eta(x_i,y_j,t_k),x_i,y_j,t_{k-1})-\hat f(u_\eta(x_i,y_j,t_{k-1}),x_i,y_j,t_{k-1})| \\
\leq &c_3|t_k-t_{k-1}|+c_2|u_\eta(x_i,y_j,t_k)-u_\eta(x_i,y_j,t_{k-1})| \\
\leq &c_3\Delta t+c_1c_2\Delta t \\
=&(c_3+c_1c_2)\Delta t,
\end{align*}
from which we obtain an approximation for the nonlinear source term $\hat f(u_\eta,x,y,t)$; i.e.,
\begin{equation}\label{f_approximation}
\hat f(u_\eta(x_i,y_j,t_k),x_i,t_j,t_k)=\hat f(u_\eta(x_i,y_j,t_{k-1}),x_i,y_j,t_{k-1})+\mO(\Delta t).
\end{equation}

Denote $c_x=\frac{\Delta tk_xc_{\alpha_1}}{h_x^{\alpha_1}}>0$, $c_y=\frac{\Delta t k_yc_{\alpha_2}}{h_y^{\alpha_2}}>0$.
By applying (\ref{left-RL}), (\ref{right-RL}), (\ref{time-discretization}) and \eqref{f_approximation}  to (\ref{M1}), we obtain the following expression{\small
\begin{equation}\label{approximation}
\begin{split}
&u_\eta(x_i,y_j,t_k)-c_x\left(\sum_{l=0}^{i+1}g_{l}^{(\alpha_1)}u_\eta(x_{i-l+1},y_j,t_k)+
\sum_{l=0}^{n_1-i+2}g_{l}^{(\alpha_1)}u_\eta(x_{i+l-1},y_j,t_k)\right) \\
&-c_y\left(\sum_{l=0}^{j+1}g_{l}^{(\alpha_2)}u_\eta(x_i,y_{j-l+1},t_k)+
\sum_{l=0}^{n_2-j+2}g_{l}^{(\alpha_2)}u_\eta(x_i,y_{j+l-1},t_k)\right)+{\Delta t}\frac{1-1_{\Omega}(x,y)}{\eta}u_\eta(x_i,y_j,t_k)\\
=&u_\eta(x_i,y_j,t_{k-1})+{\Delta t}\hat f(u_\eta(x_i,y_j,t_{k-1}),x_i,y_j,t_{k-1})+{\Delta t}r_{ij}^{k},
\end{split}
\end{equation}}
where there exists a constant $c_0$ such that
\begin{equation}\label{residual}
|r_{ij}^{k}|<c_0(h_x+h_y+\Delta t), \ 1\le i\le n_1, \ 1\le j\le n_2, \ 1\le k\le m.
\end{equation}

Define the penalization coefficients $d_{i,j}=0$ for $(x_i,y_j)\in\Omega$ and $d_{i,j}=\frac{\Delta t}{\eta}$ for $(x_i,y_j)\in\bar\Omega\setminus\Omega$. Denoting $f_{i,j}^{k}=\hat f(u_\eta(x_i,y_j,t_k),x_i,y_j,t_k)$, setting $u_{i,j}^{k}$ as the numerical approximation of $u_\eta(x_i,y_j,t_k)$ and omitting the small term $r_{i,j}^{k}$, we can construct the difference scheme for solving (\ref{M1}) with the initial and boundary conditions of \eqref{I1} and \eqref{B1} as following
{\small
 \begin{flalign}
&u_{i,j}^{k}-c_x\left(\sum_{l=0}^{i+1}g_{l}^{(\alpha_1)}u_{i-l+1,j}^k+
\sum_{l=0}^{n_1-i+2}g_{l}^{(\alpha_1)}u_{i+l-1,j}^k\right)-c_y\left(\sum_{l=0}^{j+1}g_{l}^{(\alpha_2)}u_{i,j-l+1}^{k}
+\sum_{l=0}^{n_2-j+2}g_{l}^{(\alpha_2)}u_{i,j+l-1}^k\right)& \nonumber\\  \label{difference_scheme}
&+d_{i,j}u_{i,j}^{k}=u_{i,j}^{k-1}+{\Delta t}f_{i,j}^{k-1},\quad 1\le i\le n_1, \ 1\le j\le n_2,\ 1\le k\le m, & \\ \label{I2}
&u_{i,j}^{0}=\hat u_{0}(x_i,y_j), \quad 0\leq i\leq n_1+1, \quad 0\leq j\leq n_2+1,& \\ \label{B2}
&u_{0,j}^{k}=u_{n_1+1,j}^{k}=u_{i,0}^{k}=u_{i,n_2+1}^{k}=0, \quad 1\le i\le n_1, \ 1\le j\le n_2,\ 1\leq k\leq m.&
\end{flalign}
}

Let $I$ be the identity matrix with an appropriate size and $N=n_1n_2$. Denote
$$ u^{k}=[u_{1,1}^k,\dots,u_{n_1,1}^k,u_{1,2}^k,\dots,u_{n_1,2}^k,\dots,u_{1,n_2}^k,\dots,u_{n_1,n_2}^k]^{\tT},
$$
$$ f^{k}=[f_{1,1}^k,\dots,f_{n_1,1}^k,f_{1,2}^k,\dots,f_{n_1,2}^k,\dots,f_{1,n_2}^k,\dots,f_{n_1,n_2}^k]^{\tT},
$$
$$
D={\rm diag}(d_{1,1},\dots,d_{n_1,1},d_{1,2},\dots,d_{n_1,2},\dots,d_{1,n_2},\dots,d_{n_1,n_2}).
$$
Accordingly, the difference scheme \eqref{difference_scheme}--\eqref{B2} can be expressed as the following matrix vector form
\begin{equation}\label{matrix_vector}
(I-A+D)u^{k}=u^{k-1}+{\Delta t}f^{k-1},
\end{equation}
with $D$ representing the penalization matrix and
\begin{equation}\label{toeplitz_A}
A=I_{n_2}\otimes A_{x}+A_{y}\otimes I_{n_1},
\end{equation}
where $A_x=c_xG_{n_1}^{(\alpha_1)}$ and $A_y=c_yG_{n_2}^{(\alpha_2)}$, in which
\begin{align}
 G_{n}^{(\alpha)}=
 \left[
\begin{array}{cccccc}
2g_{1}^{(\alpha)} & g_{0}^{(\alpha)}+g_{2}^{(\alpha)} & g_{3}^{(\alpha)} &\ddots &g_{n-1}^{(\alpha)} &g_{n}^{(\alpha)} \\
 g_{0}^{(\alpha)}+g_{2}^{(\alpha)} &2g_{1}^{(\alpha)} &g_{0}^{(\alpha)}+g_{2}^{(\alpha)} &g_{3}^{(\alpha)} &\ddots  &g_{n-1}^{(\alpha)}  \\
 \vdots  & g_{0}^{(\alpha)}+g_{2}^{(\alpha)} &2g_{1}^{(\alpha)} & \ddots & \ddots & \vdots \\
    \vdots  & \ddots    & \ddots & \ddots & \ddots &  g_{3}^{(\alpha)}\\
 g_{n-1}^{(\alpha)}  & \ddots & \ddots &\ddots & 2g_{1}^{(\alpha)} & g_{0}^{(\alpha)}+g_{2}^{(\alpha)} \\
g_{n}^{(\alpha)} & g_{n-1}^{(\alpha)}  & \cdots & \cdots & g_{0}^{(\alpha)}+g_{2}^{(\alpha)} &2g_{1}^{(\alpha)}\\
    \end{array}
  \right]. \label{G}
\end{align}
It is obvious that $G_n^{(\alpha)}$ is a symmetric Toeplitz matrix.
In addition, it has been shown that the entries $g_{l}^{(\alpha)}$ for $l\ge 0$ defined in (\ref{g_form}) satisfy the following properties.

\begin{lemma}\label{g-property}
{\rm(See \cite{MT-JCAM-2004})} For ${\alpha}\in(1,2)$, the coefficients $g_{l}^{(\alpha)}$, $l=0,1,\dots$, satisfy
\begin{displaymath}
\left\{
\begin{aligned}
&g_{0}^{(\alpha)}=1, \ g_{1}^{(\alpha)}=-{\alpha}<0, \ g_{2}^{(\alpha)}>g_{3}^{(\alpha)}>\cdots>0,\\
&\sum\limits_{l=0}^{\infty}g_{l}^{(\alpha)}=0,\ \sum\limits_{l=0}^{n}g_{l}^{(\alpha)}<0,\  for \ n\geq 1.
\end{aligned}
\right.
\end{displaymath}
\end{lemma}

Combining the results shown in \cite{HNS} with Lemma \ref{g-property}, we conclude that $G_n^{(\alpha)}$ is negative definite for $\alpha\in(1,2)$. Moreover, we deduce that the matrix $A$ is with the following property.
\begin{corollary}\label{corollary}
The matrix $A$ defined in \eqref{toeplitz_A} is symmetric negative definite.
\end{corollary}

For convenience in our later study, we simplify the linear system (\ref{matrix_vector}) as
\begin{equation}\label{linearsystem}
Mu^{k}=b^{k-1},
\end{equation}
where $M=I-A+D$ and $b^{k-1}=u^{k-1}+\Delta tf^{k-1}$.
Based on the results shown in Corollary \ref{corollary}, the following conclusion can be drawn reasonably.
\begin{lemma}\label{M_PD}
The coefficient matrix $M=I-A+D$ is symmetric positive definite and follows
$$
\sum_{j=1,j\ne i}^{n_1n_2}|[M]_{ij}|\le [M]_{ii}-1, \  for \ i=1,\dots,n_1n_2.
$$
\end{lemma}
\begin{proof}
From Corollary \ref{corollary}, we derive that the coefficient matrix $M$ is positive definite.
It is easy to check that $[M]_{ii}=1-[A]_{ii}+\frac{\eta}{\Delta t}$ if $(x_i,y_i)\in \bar\Omega\setminus\Omega$, or $[M]_{ii}=1-[A]_{ii}$. By following the fact that $A$ is strictly diagonal dominant with negative diagonal elements, we obtain
$$
\sum_{j=1,j\ne i}^{n_1n_2}|[M]_{ij}|=\sum_{j=1,j\ne i}^{n_1n_2}|[A]_{ij}|\le |[A]_{ii}|\le [M]_{ii}-1.
$$
\end{proof}

\section{Stability and convergence analysis}\label{S_AND_C}
In this section, the stability and convergence of the difference scheme (\ref{difference_scheme})-\eqref{B2} are discussed. Firstly, we introduce an auxiliary lemma, which plays a critical role in later investigation.
\begin{lemma}\label{infty_norm}
{\rm (See \cite{CLTA-AMC-2013})} Let $v=[v_1,v_2,\dots,v_n]^{\tT}\in \IR^{n}$ be an arbitrary vector. If the matrix $B=[b_{i,j}]_{n\times n}$ satisfies the following condition
$$
\sum_{l=1,l\neq i}^{n}|b_{i,l}|\leq b_{i,i}-1,
$$
then we have $\|v\|_{\infty}\leq\|Bv\|_{\infty}$.
\end{lemma}

In order to investigate the stability of the difference scheme, we begin with some notations. Suppose $\tilde{u}_{i,j}^{k}(1\leq i\leq n_1,\ 1\leq j\leq n_2)$ be the approximation solution of the difference scheme \eqref{difference_scheme}. Denote $\epsilon_{i,j}^{k}=\tilde{u}_{i,j}^{k}-u_{i,j}^{k}(1\leq i\leq n_1,\ 1\leq j\leq n_2)$ be the corresponding error. Then, the error vector can be depicted as
$$
\epsilon^{k}=[\epsilon_{1,1}^{k},\dots,\epsilon_{n_1,1}^{k},\epsilon_{1,2}^{k},\dots,\epsilon_{n_1,2}^{k},\dots,
\epsilon_{1,n_2}^{k},\dots,\epsilon_{n_1,n_2}^{k}]^{\tT}.
$$
Let $\gamma_{i,j}^{k}=\hat f(\tilde{u}_{i,j}^{k},x_i,y_j,t_k)-\hat f(u_{i,j}^{k},x_i,y_j,t_k)$. Denote
$$
\gamma^{k}=[\gamma_{1,1}^{k},\dots,\gamma_{n_1,1}^{k},\gamma_{1,2}^{k},\dots,\gamma_{n_1,2}^{k},\dots,
\gamma_{1,n_2}^{k},\dots,\gamma_{n_1,n_2}^{k}]^{\tT}.
$$

\begin{lemma}
The difference scheme \eqref{difference_scheme}-\eqref{B2} is unconditionally stable.
\end{lemma}
\begin{proof}
According to \eqref{difference_scheme} and \eqref{matrix_vector}, we obtain the error equation as
$$
M\epsilon^{k}=\epsilon^{k-1}+\Delta t\gamma^{k-1}.
$$
Since $\hat f(u_\eta,x,y,t)$ satisfies the Lipschitz condition, it yields
$$
|\hat f(\tilde{u}_{ij}^{k-1},x_i,y_j,t_{k-1})-\hat f(u_{ij}^{k-1},x_i,y_j,t_{k-1})|\leq c_2|\tilde{u}_{ij}^{k-1}-u_{ij}^{k-1}|=c_2|\epsilon_{ij}^{k-1}|,
$$
i.e.,
$$
\|\gamma^{k-1}\|_{\infty}\leq c_2\|\epsilon^{k-1}\|_{\infty}.
$$
Combining Lemma \ref{M_PD} with Lemma \ref{infty_norm}, we have the following inequality
$$
\|\epsilon^{k}\|_{\infty}\leq \|M\epsilon^{k}\|_{\infty}=\|\epsilon^{k-1}+\Delta t\gamma^{k-1}\|_{\infty}\leq (1+\Delta tc_2)\|\epsilon^{k-1}\|_{\infty}.
$$
By repeating the above inequality $k$ times and making use of the Gronwall inequality \cite{Gronwall-2009}, we derive
$$
\|\epsilon^{k}\|_{\infty}\leq (1+\Delta tc_2)^{k}\|\epsilon^0\|_{\infty}\leq e^{c_2T}\|\epsilon^{0}\|_{\infty},
$$
which indicates that the difference scheme is unconditionally stable.
\end{proof}


Now, we pay attention to the convergence of the difference scheme. Here some notations are presented. Let $r_{i,j}^{k}$ be the truncated error between difference scheme \eqref{difference_scheme} and equation \eqref{M1} shown in \eqref{residual}. Denote
$$
r^k=[r_{1,1}^k,\dots, r_{n_1,1}^{k},r_{1,2}^{k},\dots,r_{n_1,2}^{k},\dots,r_{1,n_2}^{k},\dots,r_{n_1,n_2}^{k}]^{\tT}.
$$
Let $\delta_{i,j}^{k}=u_\eta(x_i,y_j,t_k)-u_{i,j}^{k}$, $1\leq i\leq n_1, 1\leq j\leq n_2$, be the error between the exact solution of the problem \eqref{M1}--\eqref{B1} and the numerical solution of the difference scheme \eqref{difference_scheme}--\eqref{B2}. The error vector can be written as
$$
\delta^k=[\delta_{1,1}^k,\dots, \delta_{n_1,1}^{k},\delta_{1,2}^{k},\dots,\delta_{n_1,2}^{k},\dots,\delta_{1,n_2}^{k},\dots,\delta_{n_1,n_2}^{k}]^{\tT}.
$$
Set $\xi_{i,j}^{k}$($1\leq i\leq n_1, 1\leq j\leq n_2$) be the error between $\hat f(u_\eta(x_i,y_j,t_k),x_i,y_j,t_k)$ and $\hat f(u_{i,j}^{k},x_i,y_j,t_k)$.
Denote
$$
{\xi}^k=[{\xi}_{1,1}^k,\dots,{\xi}_{n_1,1}^{k},
\xi_{1,2}^{k},\dots,\xi_{n_1,2}^{k},\dots,\xi_{1,n_2}^{k},\dots,\xi_{n_1,n_2}^{k}]^{\tT}.
$$
Then, we have the following convergent theorem.

\begin{lemma}
The difference scheme defined in \eqref{difference_scheme}--\eqref{B2} is convergent and satisfies
$$
\|\delta^k\|_{\infty}\leq c(\Delta t+h_x+h_y),
$$
where $c$ is a positive constant independent of temporal step ${\Delta}t$ and spatial step $h_x$ and $h_y$.
\end{lemma}
\begin{proof}
Since $\hat f(u_\eta,x,y,t)$ satisfies the Lipschitz condition concerning $u_{\eta}$, we have
$$
|\xi_{i,j}^{k}|\leq c_2|\delta_{i,j}^{k}|, \  for \ all \ i,j.
$$
Then, it follows that
$$
\|\xi^k\|_{\infty}\leq c_2\|\delta^{k}\|_{\infty}.
$$
Due to the difference scheme \eqref{difference_scheme} is consistent, according to \eqref{approximation} and \eqref{matrix_vector}, we obtain the following error equation
$$
M\delta^k=\delta^{k-1}+\Delta t\xi^{k-1}+\Delta tr^k,
$$
where $\delta^0=0$ and $\|r^k\|_{\infty}\leq c_0(\Delta t+h_x+h_y)$.
By virtue of Lemma \ref{M_PD} and Lemma \ref{infty_norm} again, it leads to
\begin{align*}
\|\delta^k\|_{\infty}&\leq \|M\delta^k\|_{\infty}\\
&\leq\|\delta^{k-1}\|_{\infty}+\Delta t \|\xi^{k-1}\|_{\infty}+\Delta t\|r^k\|_{\infty}\\
&\leq(1+\Delta tc_2)\|\delta^{k-1}\|_{\infty}+\Delta t c_0(\Delta t+h_x+h_y).
\end{align*}
By repeating the above processes $k$ times, it follows that
$$
\|\delta^k\|_{\infty}\leq \frac{c_0}{c_2}(1+\Delta tc_2)^k(\Delta t+h_x+h_y)\leq \frac{c_0}{c_2}e^{c_2T}(\Delta t+h_x+h_y)=c(\Delta t+h_x+h_y),
$$
where $c=\frac{c_0}{c_2}e^{c_2T}$.
Therefore, we confirm that the difference method is convergent.
\end{proof}

\section{Implementation}\label{implement}

In this section, we expect to numerically solve the linear system (\ref{linearsystem}). As the coefficient matrix of the system is with the Toeplitz-like structure, the Krylov subspace method is employed to solve the linear system. In order to speed up the convergence rate of the iterative method, an efficient preconditioner is indispensable. In the following, we concentrate on constructing a preconditioner and discussing the spectrum of the preconditioned matrix.


\subsection{$\tau$ preconditioner}

Firstly, we recall the sine transform based preconditioner, which is also called the ${\tau}$ preconditioner. Denote $T_n=[t_{|i-j|}]_{n\times n}$ be an $n\times n$ symmetric Toeplitz matrix. Then, the corresponding $\tau$ preconditioner of $T_n$ can be determined by the Hankel correction \cite{BB-Conference-1990}. More precisely, the $\tau$ matrix can be described as
\begin{equation}
{\tau}(T_n)=T_n-H_n,
\end{equation}
where $H_n$ is a Hankel matrix whose entries are constant along the antidiagonals, in which the antidiagonals are depicted as
$$
[t_2,t_3,\dots,t_{n-1},0,0,0,t_{n-1},\dots,t_3,t_2]^{\tT}.
$$

Remark that the ${\tau}$ matrix can be diagonalized by the sine transform matrix, which is written as
$$
{\tau}(T_n)=S_n\Lambda_nS_n,
$$
where the diagonal matrix $\Lambda_n$ is consist of all the eigenvalues of the matrix ${\tau}(T_n)$, and $S_n$ is a symmetric orthogonal matrix whose entries are given by
\begin{equation*}
[S_n]_{i,j}=\sqrt{\frac{2}{n+1}}\sin{(\frac{\pi ij}{n+1})}, \quad 1\leq i,j \leq n.
\end{equation*}
Then, the matrix vector multiplication $S_nv$ for any vector $v$ can be done by the discrete sine transform and only $\mO(n\log n)$ operations are required. The eigenvalues of the $\tau$ matrix are determined by its first column with $\mO(n)$ storage being needed. In the following, we construct the preconditioner for system \eqref{linearsystem}.

Recall the coefficient matrix, the corresponding preconditioner based on the sine transform are described as
\begin{equation}\label{p1}
P=I-\tau_1(A)+D,
\end{equation}
where
\begin{equation}\label{tau_A}
\tau_1(A)=I_{n_2}\otimes \tau(A_x)+\tau(A_y)\otimes I_{n_1}.
\end{equation}
However, the diagonal matrix $D$ cannot be diagonalized by the sine transform matrix, which leads to that $P$ is hard to be inverted in $\mO(N\log N)$ operations. We then follow the idea in \cite{DSW-CAM-2019} to construct a workable and efficient preconditioner.

First of all, rewrite the coefficient matrix $M$ as the following splitting form
\begin{equation}
M=(I-\Phi_{d})(I-A)+\Phi_{d}\left((1+\frac{\Delta t}{\eta})I-A\right),
\end{equation}
where $\Phi_d={\rm diag}(\phi_{11},\dots,\phi_{n_1,1},\phi_{12},\dots,\phi_{n_1,2},\dots,\phi_{1,n_2},\dots,\phi_{n_1,n_2})$
is a diagonal matrix with entries
\begin{align*}
\phi_{ij}=\left\{
\begin{array}{cc}
0, &   (x_i,y_j)\in\Omega, \\
1, & \quad \ (x_i,y_j)\in\bar{\Omega}\setminus\Omega.
\end{array}
\right.
\end{align*}
Accordingly, the preconditioner can be expressed as
\begin{equation}
P=(I-\Phi_{d})(I-\tau_1(A))+\Phi_{d}\left((1+\frac{\Delta t}{\eta})I-\tau_1(A)\right).
\end{equation}
Likewise, it is too expensive to compute $P^{-1}v$  for an arbitrary vector $v$.
Therefore, we   construct a preconditioner $\hat{P}$ to approximate $P$ such that
\begin{equation}\label{preconditioner}
\hat{P}^{-1}=(I-\Phi_{d})(I-\tau_1(A))^{-1}+\Phi_{d}\left((1+\frac{\Delta t}{\eta})I-\tau_1(A)\right)^{-1}.
\end{equation}
In this circumstance, the product of the matrix $\hat P^{-1}$ and a vector can be done in $\mO(N\log N)$ operations by the discrete sine transform, and hence the computational cost on each time iteration keeps $\mO(N\log N)$ by the preconditioned Krylov subspace method. Due to the diagonal matrices destroy the symmetric structure of the coefficient matrix, the generalized minimum residual (GMRES) method with the preconditioner $\hat P$ is exploited to solve the linear system \eqref{linearsystem}. In the following, we discuss the spectrum of the preconditioned matrix.

\subsection{Spectral analysis}
For the study of the spectra, it suffices to consider the existence of the preconditioner. A useful result emerged in \cite{HNS} will be utilized to demonstrate the invertibility of the matrices $P$ and $\hat P$.


\begin{lemma}\label{A_xND}
{\rm(See \cite{HNS})} Let $G_{n_1}^{(\alpha_1)}$ and $G_{n_2}^{(\alpha_2)}$ be Toeplitz matrices as defined in \eqref{G}. Then the matrices $\tau(A_x)=c_x\tau(G_{n_1}^{(\alpha_1)})$ and $\tau(A_y)=c_y\tau(G_{n_2}^{(\alpha_2)})$ are both negative definite.
\end{lemma}

In light of the results shown in Lemma \ref{A_xND}, we derive that all the eigenvalues of $\tau(A_x)$ and $\tau(A_y)$ are less than $0$ and hence we have the following lemma.

\begin{lemma}\label{P_spd}
The matrix $P$ defined in (\ref{p1}) is positive definite and so $P$ is invertible.
\end{lemma}
\begin{proof}
Let $\Lambda_x$ and $\Lambda_y$ be the diagonal matrices including all the eigenvalues of $\tau(A_x)$ and $\tau(A_y)$, respectively. Then, we have
$$
I-\tau_1(A)=(S_{n_2}\otimes S_{n_1})(I-I_{n_2}\otimes\Lambda_x-\Lambda_y\otimes I_{n_1})(S_{n_2}\otimes S_{n_1}).
$$
By the conclusion stemming from Lemma \ref{A_xND}, we deduce that all the eigenvalues of $I-\tau_1(A)$ are positive and hence $I-\tau_1(A)$ is positive definite. Therefore, the matrix $P\geq I-\tau_1(A)$ is positive definite and so $P$ is invertible. The proof is complete.
\end{proof}

Evoking the results from Lemma \ref{P_spd}, the invertibility of the preconditioner $\hat P$ is determined in the following lemma.
\begin{lemma}
The preconditioner $\hat P$ defined in \eqref{preconditioner} is invertible.
\end{lemma}
\begin{proof}
As{\small
\begin{align*}
&\hat P^{-1}(I-\tau_1(A))\left((1+\Delta t/\eta)I-\tau_1(A)\right)\\
=&(I-\Phi_d)\left((1+\Delta t/\eta)I-\tau_1(A)\right)+\Phi_d\left((1+\Delta t/\eta)I-\tau_1(A)\right)^{-1}(I-\tau_1(A))\left((1+\Delta t/\eta)I-\tau_1(A)\right)\\
=&(I-\Phi_d)\left((1+\Delta t/\eta)I-\tau_1(A)\right)+\Phi_d(I-\tau_1(A)) \\
=&I-\tau_1(A)+\Delta t/\eta(I-\Phi_d)\\
\geq& I-\tau_1(A),
\end{align*}}
we certify that $\hat P^{-1}$ is invertible owing to the fact that both $I-\tau_1(A)$ and $(1+\Delta t/\eta)I-\tau_1(A)$ are positive definite, from which we confirm that $\hat P$ is invertible.
\end{proof}

 Next, we focus on the spectrum of the matrix $P^{-1}M$. Some conclusions are proposed for our later investigation.

\begin{lemma}\label{eigenvalue_A}
{\rm (See \cite{HNS})} The matrix $A$ is a block Toeplitz with Toeplitz block matrix defined in \eqref{toeplitz_A}. $\tau_1(A)$ is the corresponding block Toeplitz with $\tau$ block matrix defined in \eqref{tau_A}. Then, the spectrum of $\tau_1(A)^{-1}A$ are uniformly bounded in the open interval $(1/2,3/2)$.
\end{lemma}

Based on the results shown in Lemma \ref{eigenvalue_A}, we have the following results.


\begin{lemma}\label{spectrum-BM}
The spectrum of the matrix $P^{-1}M$ are uniformly bounded below by $\frac{1}{2}$ and bounded above by $\frac{3}{2}$.
\end{lemma}

\begin{proof}
Let $x\in\IR^{N}$ be an arbitrary vector. By the Rayleigh quotients theorem and the conclusion presented in Lemma \ref{eigenvalue_A}, it holds that
$$
\frac{1}{2}\leq \frac{x^*Ax}{x^*\tau_1(A)x}\leq \frac{3}{2}.
$$
It immediately follows that
$$
\min\limits_{x}\left\{1,\frac{x^*Ax}{x^*\tau_1(A)x}\right\}<\frac{x^*(I-A+D)x}{x^*(I-\tau_1(A)+D)x}=
\frac{x^*(I+D)x-x^*Ax}{x^*(I+D)x-x^*\tau_1(A)x}<\max\limits_{x}\left\{1,\frac{x^*Ax}{x^*\tau_1(A)x}\right\}.
$$
Therefore, we derive
$$
\lambda_{\min}(P^{-1}M)=\min\limits_{x}\frac{x^{*}(I-A+D)x}{x^*(I-\tau_1(A)+D)x}>\frac{1}{2}, \
\lambda_{\max}(P^{-1}M)=\max\limits_{x}\frac{x^{*}(I-A+D)x}{x^{*}(I-\tau_1(A)+D)x}<\frac{3}{2}.
$$
The proof is complete.
\end{proof}

This lemma implies that when apply the matrix $P$ as the preconditioner, the spectrum of the preconditioned matrix is uniformly bounded and hence the GMRES method converges linearly. From the theoretical point of view, the matrix $P$ is worth to be treated as the preconditioner. However, from the perspective of practical computation, computing the inverse of $P$ requires more computational cost and storage as mentioned before, which does not comply with our original intention.
Actually, in our practical operation, the matrix $\hat P$ is utilized to accelerate the convergence rate.
\begin{remark}
We wish to establish a result similar to the one for the spectral of $P^{-1}M$, which provides an upper bound and a lower bound. Unfortunately, it is helpless to prove a conclusion such as Lemma \ref{spectrum-BM}. Up to now, there is no any theoretical analysis for the spectrum of the preconditioned matrix $\hat P^{-1}M$. Nevertheless, it is remarkable to mention that such kind of preconditioner has been used to several model equations. The numerical results indicate that the preconditioner $\hat{P}$ is efficient and feasible.
Therefore, we continually make use of the preconditioner under the case without theoretical support.
\end{remark}

\section{Numerical results} \label{numerical results}
In this section, the numerical experiments are carried out to demonstrate the effectiveness of the proposed method. The GMRES method and the preconditioned GMRES method are applied to solve the linear system (\ref{matrix_vector}), respectively. Set the restart number be $20$ and the stopping criterion of those methods as
$$
\frac{\|r^{(k)}\|_2}{\|r^{(0)}\|_2}<10^{-8},
$$
where $r^{(k)}$ means the residual vector after $k$ iterations.
 The initial guess is chosen as
\begin{align*}
v_0=\left\{
\begin{array}{lc}
u^{0},\quad \qquad  & m=0, \\
u^{m+1}=2u^m-u^{m-1}, & m>0.
\end{array}
\right.
\end{align*}

To exhibit the performance of the proposed method, the ADI method and the preconditioned ADI method \cite{CLTA-ANM-2018} are implemented as comparisons. In the following tables, `PGMRES' and `PADI' represent the preconditioned GMRES method and peconditioned ADI method, respectively. `Iter' means the average number of iterations by those iterative methods. `CPU(s)' displays the total CPU time in seconds for solving the whole discretized system.
In addition, the `Error' denotes the infinite norm of the relative error between the exact solution and the numerical solution on the original area $\Omega$ as
$$
Error=\frac{\|u_e-u_\eta\|_{\infty}}{\|u_e\|_{\infty}},
$$
where $u_e$ is the exact solution and $u_\eta$ is the numerical solution. All numerical results are carried out by MATLAB R2017a on a dell PC with the configuration: Intel(R)
Core(TM)i7-8700 CPU @3.20 3.20GHz and 8 GB RAM.

\bigskip
\begin{example}\label{example}
{\rm (See \cite{CLTA-ANM-2018})} Consider the RSFDEs defined on the following elliptical domain:
\begin{equation*}
(x,y)\in{\Omega}=\{(x,y)|(x-a)^{2}/a^{2}+(y-b)^{2}/b^{2}\le 1\},
\end{equation*}
with the initial condition:
$$
u(x,y,t)=((x-a)^2/a^2+(y-b)^2/b^2-1)^2,\quad (x,y)\in {\Omega},
$$
and the zero Dirichlet boundary condition:
$$
u(x,y,t)=0,\quad (x,y)\in \partial \Omega.
$$
The exact solution is given by $u(x,y,t)=e^{-t}((x-a)^2/a^2+(y-b)^2/b^2-1)^2$ and the source term is depicted as
\begin{align*}
f(u,x,y,t)=&k_xc_{\alpha_1}e^{-t}a^4[h(\alpha_1,x-a+c_y,c_y)+h(\alpha_1,a+c_y-x,c_y)] \\
&+k_yc_{\alpha_2}e^{-t}b^4[h(\alpha_2,y-b+c_x,c_x)+h(\alpha_2,b+c_x-y,c_x)]-u(x,y,t),
\end{align*}
where\ {\small
$h(\alpha,s,d)=\frac{24s^{(4-\alpha)}}{\Gamma(5-\alpha)}-\frac{24ds^{(3-\alpha)}}{\Gamma(4-\alpha)}+
\frac{8d^{2}s^{(3-\alpha)}}{\Gamma(3-\alpha)}, c_y=a\sqrt{(1-(y-b)^2/b^2)}, c_x=b\sqrt{(1-(x-a)^2/a^2)}$}.
\end{example}

\begin{table}[t]
\begin{center}
\caption{Comparisons for solving Example 1 by the GMRES method, the PGMRES method, the ADI method and the PADI method for different coefficients.}\label{table1}
\def\temptablewidth{1\textwidth}
{\rule{\temptablewidth}{1pt}}
{\footnotesize
\begin{tabular*}{\temptablewidth}{@{\extracolsep{\fill}}cccccccccccc}
      &         &  \multicolumn{3}{c}{$GMRES$} & \multicolumn{2}{c}{$PGMRES$} & \multicolumn{3}{c}{$ADI$} & \multicolumn{2}{c}{$PADI$}\\
 \cline{3-5}\cline{6-7}\cline{8-10}\cline{11-12}
 $k_x$ & $n_1$  &Error  &Iter  &CPU(s)  &Iter  &  CPU(s) & Error &Iter & CPU(s)  & Iter & CPU(s)   \\ \hline
$10^{-4}$&$2^5$ &1.58e-2 & 3.06 & 0.05  & 1.03 & 0.05 &1.58e-2 &3.00&0.15   &3.00&0.17\\
         &$2^6$ &7.90e-3 &1.80 & 0.13   & 1.01 & 0.22 &7.90e-3 &3.00&0.77   &2.00&0.54\\
         &$2^7$ &3.90e-3 & 1.76& 0.57   & 1.00 & 1.47 &3.90e-3 &2.00&2.60   &2.00&2.99\\
         &$2^8$ &2.00e-3 & 1.75& 6.84   & 1.00 &15.34 &2.00e-3 &2.00&18.01  &2.00&20.38 \\
         &$2^9$ &9.77e-4 & 1.74& 62.29  & 1.00&131.57&9.77e-4 &2.00&114.07&2.00&138.08  \\
\hline
$10^{-2}$&$2^5$ &1.51e-3 & 5.50 & 0.06  & 2.03  &0.07  &1.51e-2 &7.00&0.15  &4.00&0.17 \\
         &$2^6$ &7.50e-3 & 4.08 & 0.21  & 2.02  &0.29  &7.50e-3 &7.00&0.78  &4.00&0.74  \\
         &$2^7$ &3.80e-3 & 3.78 & 0.92  & 1.04  &1.38  &3.80e-3 &7.00&4.24  &4.00&3.89 \\
         &$2^8$ &1.90e-3 & 2.00 & 7.39  & 1.02  &14.42 &1.90e-3 &6.00&27.97 &4.00&29.44 \\
         &$2^9$ &9.40e-4 & 2.23 & 71.65 & 1.01  &130.71&9.41e-4 &6.99&193.91&4.00&194.96 \\
\hline
1 & $2^5$ & 7.90e-3  &41.19 & 0.27  &6.06  &0.11  & 1.29e-2 & 41.94 & 0.59 &11.00&0.29\\
  & $2^6$ & 4.10e-3  &42.16 & 1.91  &5.34  &0.57  & 7.20e-3 & 53.00 & 3.58 &10.03&1.58\\
  & $2^7$ & 2.10e-3  &40.80 & 8.11  &4.23  &3.28  & 3.80e-3 & 63.00 & 23.50&10.00&8.10\\
  & $2^8$ & 1.10e-3  &34.62 & 57.11 &3.10  &25.58 & 2.00e-3 & 75.00&204.21 &10.00&53.10\\
  & $2^9$ & 5.45e-4  &22.04 &493.11 &2.53  &225.19& 9.95e-4 &88.00&1.52e+3 &10.00&364.16
\end{tabular*} {\rule{\temptablewidth}{1pt}}
}
\end{center}
\end{table}

\captionsetup{width=0.7\textwidth}
\begin{table}[t]       
\begin{center}
\caption{Comparisons for solving Example 1 by the PGMRES method with $\tau$-based preconditioner and PADI method with circulant preconditioner for $T=10$.}\label{table2}
\def\temptablewidth{0.7\textwidth}
{\rule{\temptablewidth}{1pt}}
{\footnotesize
\begin{tabular*}{\temptablewidth}{@{\extracolsep{\fill}}cccccccccccc}
        &  \multicolumn{3}{c}{$PGMRES$} & \multicolumn{3}{c}{$PADI$} \\
                  \cline{2-4}\cline{5-7}
$n_1$    &Error   &Iter  &CPU(s)  & Error  &Iter   & CPU(s)      \\ \hline
$2^5$    &7.90e-3 & 5.11 & 1.00   &1.35e-2 & 11.00 & 3.16     \\
$2^6$    &4.10e-3 & 5.03 & 5.66   &7.60e-3 & 10.01 & 16.54     \\
$2^7$    &2.10e-3 & 4.02 & 30.71  &4.00e-3 & 10.00 & 79.70      \\
$2^8$    &1.10e-3 & 3.01 & 254.71 &2.10e-3 & 10.00 &530.94     \\
$2^9$    &5.46e-4 & 2.58 & 2.17+3 &1.00e-3 & 10.00 &3.65e+3     \\
\hline
\end{tabular*} {\rule{\temptablewidth}{1pt}}
}
\end{center}
\end{table}

In this example, we extend $\Omega$ to be a square domain $\bar\Omega=(-a,a)\times(-b,b)$. Assume that the values of $\hat f$ and $\hat u_0$ on the extended region $\bar\Omega\setminus\Omega$ are both $0$. In the following tables, take $\alpha_1=1.4$ and $\alpha_2=1.7$,
 $a=2$ and $b=1$. Let $k_x=k_y$ and $n_1=n_2$. Table \ref{table1} shows the numerical results of Example \ref{example} with $T=1$, $m=n_1$, and $\eta=10^{-5}$. From this table, we see that all mentioned methods implement well when the diffusion coefficients $k_x$ and $k_y$ are small. In particular, for those cases with no preconditioners are still powerful and even superior than the cases with preconditioners from the consumed CPU times point of view.
The reason of this phenomenon is that the model equations become time direction dominant if the diffusion coefficients are small, which yields that the coefficient matrix is almost equivalent to the identity matrix. In this situation, the preconditioner seems superfluous and thus more computational cost and memory are required. Nevertheless, the GMRES method still works well compared with both the ADI method and the preconditioned ADI method from the perspective of the required iterations and CPU times.
Moreover, the merit of the preconditioner is shown under the cases where the coefficients become large.  We see that when $k_x=k_y=1$, only few iterations and CPU times are required for the preconditioned GMRES method to attain convergence.

On the other hand, the advantage of our proposed method is more evident provided that the time $T$ becomes large as shown in Table \ref{table2}, where we take $T=10$, $m=10n_1$, and $k_x=1$.
It is remarkable to notice that our method is more accurate than the ADI method from the extent of the numerical approximation. With the same matrix size, the error arising from the ADI method is nearly twice than the error generating from our method. With the case of same error, the CPU time required for the ADI method to converge is much greater than the CPU time required for the proposed method. In conclusion, the proposed method should be a good choice for handling convex domain problems and more suitable for general cases.

In addition, Table \ref{table3} lists the numerical solutions of the penalized equation defined in \eqref{M1} on the extended region $\bar\Omega\setminus\Omega$ for different parameter $\eta$. From Table \ref{table3}, we obverse that the solutions tend to $0$ as the parameter is closed to $0$ and the values of the solutions are rely on the value of the parameter, which is coincident with the fact that the solution of the penalized equation converges to the solution of the original equation. Therefore, we have demonstrated the effectiveness of the proposed method.

\captionsetup{width=0.7\textwidth}
\begin{table}[t]       
\begin{center}
\caption{Values of $\|u_{\eta}\|_{\infty}$ for
 $k_x=10^{-4}$ in the extended region $\bar\Omega\setminus\Omega$ with different penalization parameters $\eta=10^{-4}, 10^{-5}, 10^{-6}$.}\label{table3}
\def\temptablewidth{0.7\textwidth}
{\rule{\temptablewidth}{1pt}}
{\footnotesize
\begin{tabular*}{\temptablewidth}{@{\extracolsep{\fill}}cccc}
$n_1$    &$\eta=10^{-4}$   &$\eta=10^{-5}$  &$\eta=10^{-6}$        \\ \hline
$2^5$    &6.0224e-9      & 6.0218e-10   &6.0218e-11          \\
$2^6$    &7.0861e-9      & 7.0855e-10   &7.0854e-11          \\
$2^7$    &8.4002e-9      & 8.3996e-10   &8.3995e-11          \\
$2^8$    &8.3910e-9      & 8.3907e-10   &8.3907e-11          \\
$2^9$    &8.5486e-9      & 8.5487e-10   &8.5487e-11          \\
$2^{10}$ &8.8833e-9      & 8.8847e-10   &8.8848e-11          \\
\hline
\end{tabular*} {\rule{\temptablewidth}{1pt}}
}
\end{center}
\end{table}

\section{Concluding remarks} \label{conclusion}
In this paper, the two-dimensional Riesz space fractional diffusion equations defined in convex domains are considered. We employ the volume-penalized method to convert the convex domain to be a rectangular domain. After that, an implicit difference method is utilized to discretize the reformulate equations with penalized term, which generates a block Teoplitz with Toeplitz block matrix plus a diagonal matrix. The GMRES method is applied to solve the discrete linear system. In quest to accelerate the convergence rate, a preconditioner based on the sine transform matrix is constructed and the spectra of the preconditioned matrix are investigated.
The numerical results demonstrate that our proposed method is implementable and efficient.
In our future investigation, we devote to improving the convergent order of the difference scheme based on the penalization method. Besides, we pay attention to solve high dimensional Riesz fractional diffusion equations defined on irregular domains.

\bibliographystyle{unsrt}

\end{document}